\documentclass{amsart}
\usepackage[american]{babel}
\usepackage[T1]{fontenc}
\usepackage[utf8]{inputenc}
\usepackage{amsmath,amsfonts,amstext, dsfont}
\usepackage{amsthm,amssymb,amsmath,amscd,graphicx,epsfig,hhline,mathrsfs,latexsym,url}
\usepackage[neverdecrease]{paralist}
\usepackage{yfonts,dsfont}
\usepackage{tikz}
\usetikzlibrary{shapes.geometric, arrows}
\newcommand{\mc}{\mathcal}
\newcommand{\ms}{\mathscr}

\newcommand{\mf}{\mathbf}
\newcommand{\mb}{\mathbb}
\newcommand{\mr}{\mathrm}
\newcommand{\aveN}{\frac{1}{N}\sum_{n=1}^N}

\newtheorem{Thm}{Theorem}[section] 
\newtheorem*{Rem}{Remark}

\newtheorem{Le}[Thm]{Lemma}
\newtheorem{Cor}[Thm]{Corollary} 
\newtheorem{Prop}[Thm]{Proposition}
\theoremstyle{definition}
\newtheorem{Def}{Definition}[section]
\theoremstyle{plain}

\newcommand{\lin}{\text{lin}}

\title{Pointwise entangled ergodic theorems for Dunford-Schwartz operators}

\author{D\'avid Kunszenti-Kov\'acs}
\address{MTA Alfréd Rényi Institute of Mathematics, P.O. Box 127, H-1364 Budapest, Hungary}
\email{daku@renyi.hu}

\thanks{The author has received funding from the European Research Council under the European Union's Seventh Framework Programme (FP7/2007-2013) / ERC grant agreement $\mr{n}^\circ$617747, and from the MTA R\'enyi Institute Lend\"ulet Limits of Structures Research Group.}

\begin{document}

\maketitle

\begin{abstract}

We investigate pointwise convergence of entangled ergodic averages of Dunford-Schwartz operators $T_0,T_1,\ldots, T_m$ on a Borel probability space. These averages take the form
\[
\frac{1}{N^k}\sum_{1\leq n_1,\ldots, n_k\leq N} T_m^{n_{\alpha(m)}}A_{m-1}T^{n_{\alpha(m-1)}}_{m-1}\ldots A_2T_2^{n_{\alpha(2)}}A_1T_1^{n_{\alpha(1)}} f,
\]
where $f\in L^p(X,\mu)$ for some $1\leq p<\infty$, and $\alpha:\left\{1,\ldots,m\right\}\to\left\{1,\ldots,k\right\}$ encodes the entanglement. We prove that under some joint boundedness and twisted compactness conditions on the pairs $(A_i,T_i)$, almost everywhere convergence holds for all $f\in L^p$. We also present an extension to polynomial powers in the case $p=2$, in addition to a continuous version concerning Dunford-Schwartz $C_0$-semigroups.
\end{abstract}


\section{Introduction}

Entangled ergodic averages were first introduced in a paper by Accardi, Hashimoto and Obata \cite{AHO}, where these were a key ingredient in providing an analogue of the Central Limit Theorem for the models in quantum probability they studied. Entangled ergodic averages take the general form
\[
\frac{1}{N^k}\sum_{1\leq n_1,\ldots, n_k\leq N} T_m^{n_{\alpha(m)}}A_{m-1}T^{n_{\alpha(m-1)}}_{m-1}\ldots A_2T_2^{n_{\alpha(2)}}A_1T_1^{n_{\alpha(1)}},
\]
where $A_i$ ($1\leq i\leq m-1$) and $T_i$ ($1\leq i\leq m$) are operators on a Banach space $E$, and $\alpha:\left\{1,\ldots,m\right\}\to\left\{1,\ldots,k\right\}$ is a surjective map. The operators $A_i$ are acting as transitions between the actions of the operators $T_i$,  that iteratively govern the dynamics, whereas the entanglement map $\alpha$ provides a coupling between the stages.

Further papers on the subject initially focused on strong convergence of these Cesàro averages, see Liebscher \cite{liebscher:1999}, Fidaleo \cite{fidaleo:2007,fidaleo:2010,fidaleo:2014} and Eisner, K.-K. \cite{EKK}.\\
In Eisner, K.-K. \cite{EKK2} and K.-K. \cite{KK}, attention was turned to pointwise almost everywhere convergence in the context of the $T_i$'s being operators on function spaces $E=L^p(X,\mu)$ ($1\leq p<\infty$), where $(X,\mu)$ is a standard probability space (i.e. a compact metrizable space with a Borel probability measure). The former paper concerns itself with the case $k=1$ with the $T_i$ being Dunford-Schwartz operators, whereas the latter allows for multi-parameter entanglement, but at the price of only dealing with Koopman operators.\\
In this paper we deal with the full case of general entanglement maps $\alpha$ and Dunford-Schwartz operators $T_i$, and show a.e. convergence on the whole $L^p$ space for all $1\leq p<\infty$, significantly improving on previous results. We introduce a formalism for the iterated function splittings used in the proofs in order to make them more concise, better highlighting what the main steps are, and where the different assumptions of the statements come into play. We also provide results concerning polynomial and time-continuous versions of the ergodic theorems considered.\\
Note that in what follows, $\mathbb{N}$ will be used to denote the set of positive integers.

Our main result is as follows.

\begin{Thm}\label{thm:main}
Let $m>1$ and $k$ be positive integers, $\alpha:\left\{1,\ldots,m\right\}\to\left\{1,\ldots,k\right\}$ a not necessarily surjective map, and let $T_1,T_2,\ldots T_m$ be Dunford-Schwartz operators 
 on a Borel probability space $(X,\mu)$.
Let $p\in[1,\infty)$, $E:=L^p(X,\mu)$ and let $E=E_{j;r}\oplus E_{j;s}$ be the Jacobs-Glicksberg-deLeeuw decomposition corresponding to $T_j$ $(1\leq j\leq m)$. Let further $A_j\in\mc{L}(E)$ $(1\leq j< m)$ be bounded operators. For a function $f\in E$ and an index $1\leq j\leq m-1$, write $\ms{A}_{j,f}:=\left\{A_jT_j^nf\left|\right.n\in\mb{N}\right\}$. Suppose that the following conditions hold:
\begin{itemize}
\item[(A1)](Twisted compactness)
For any function $f\in E$, index $1\leq j\leq m-1$ and $\varepsilon>0$, there exists a decomposition $E=\mc{U}\oplus \mc{R}$ with $0<\dim \mc{U}<\infty$ 
such that
\[
P_\mc{R}\ms{A}_{j,f}
\subset B_\varepsilon(0,L^\infty(X,\mu)),
\]
with $P_\mc{R}$ denoting the projection along $\mc{U}$ onto $\mc{R}$.
\item[(A2)](Joint $\mc{L}^\infty$-boundedness)
There exists a constant $C>0$ such that we have
\[\{A_jT^n_j|n\in\mb{N},1\leq j\leq m-1\}\subset B_C(0,\mc{L}(L^\infty(X,\mu)).
\]
\end{itemize}
Then we have the following:
\begin{enumerate}
\item for each $f\in E_{1;s}$, 
\[
\frac{1}{N^k}\sum_{1\leq n_1,\ldots, n_k\leq N} \left|T_m^{n_{\alpha(m)}}A_{m-1}T^{n_{\alpha(m-1)}}_{m-1}\ldots A_2T_2^{n_{\alpha(2)}}A_1T_1^{n_{\alpha(1)}} f \right|\rightarrow 0
\]
pointwise a.e.;
\item
for each $f\in E_{1;r}$, 
\[
\frac{1}{N^k}\sum_{1\leq n_1,\ldots, n_k\leq N} T_m^{n_{\alpha(m)}}A_{m-1}T^{n_{\alpha(m-1)}}_{m-1}\ldots A_2T_2^{n_{\alpha(2)}}A_1T_1^{n_{\alpha(1)}} f
\]
converges pointwise a.e..
\end{enumerate}
\end{Thm}

\begin{Rem}
Note that it was proven in \cite{EKK2} that the Volterra operator $V$ on $L^2([0,1])$ defined through
\[
(Vf)(x):=\int_{0}^x f(z) \mr{d}z
\]
as well as all of its powers can be decomposed into a finite sum of operators, each of which satisfy conditions (A1) and (A2) when paired with any Dunford-Schwartz operator. Hence the conclusions of Theorem \ref{thm:main} apply whenever the operators $A_i$ are chosen to be powers of $V$.
\end{Rem}


\section{Notations and tools}\label{sec:prelim}

Before proceeding to the proof of our main result, we need to clarify some of the notions used, and introduce notations that will simplify our arguments.

In what follows, $\mb{N}$ will denote the set of positive integers, and $\mb{T}$ the unit circle in $\mb{C}$.

The proof works by iteratively splitting the functions into finitely many parts, and so introducing vector indices will be very helpful. Given a vector $v\in\mb{N}^c$ ($c\geq 1$), let $\overline v\in\mb{N}^{c-1}$ be the vector obtained by deleting its last coordinate, and let $v^*$ denote its last coordinate. Also, we shall write $l(v):=c$ to denote the number of coordinates of the vector, $x\subset v$ if there exist vectors $w_0,w_1,\ldots,w_b$ ($b\geq 1$) such that $w_0=x$, $w_b=v$ and for each $1\leq i\leq b$ we have $w_{i-1}=\overline{w_i}$, and finally $x\subseteq v$ if $x=v$ or $x\subset v$.

Let $\mc{N}$ denote the set of all bounded sequences $\{a_n\}\subset \ell^\infty(\mb{C})$ satisfying
$$
\lim_{N\to\infty}\aveN |a_n|=0.
$$
By the Koopman-von Neumann lemma, see e.g.~Petersen \cite[p. 65]{petersen:1983}, $(a_n)\in\mc{N}$ if and only if it lies in $\ell^\infty$ and converges to $0$ along a sequence of density $1$.\\

\begin{Def}
Given a Banach space $E$ and an operator $T\in\mc{L}(E)$, the operator $T$ is said to have \emph{relatively weakly compact orbits} if for each $f\in E$ the orbit set $\left\{ T^nf|n\in\mb{N}^+\right\}$ is relatively weakly closed in $E$. For any such operator, there exists a corresponding \emph{Jacobs-Glicksberg-deLeeuw} decomposition of the form (cf. \cite[Theorem II.4.8]{eisner-book})
$$E=E_r \oplus E_s,$$
where
\begin{eqnarray*}
E_r&:=&\overline\lin\{f\in E:\ Tf=\lambda f\mbox{ for some } \lambda\in\mb{T}\},\\
E_s&:=&\{f\in E:\ (\varphi( T^nf))\in\mc{N} \mbox{ for every } \varphi\in E'\}.
\end{eqnarray*}
\end{Def}

Note that every power bounded operator on a reflexive Banach space has relatively weakly compact orbits. Thus the above decomposition is valid for, e.g., every contraction on $L^p(X,\mu)$ for $p\in (1,\infty)$.
In addition, if $T$ is a \emph{Dunford-Schwartz operator} on $L^1(X,\mu)$, i.e., $\|T\|_1\leq 1$ and $T$ is also a contraction on $L^\infty(X,\mu)$,
then $T$ has relatively weakly compact orbits as well, see Lin, Olsen, Tempelman \cite[Prop.~2.6]{LOT} and Kornfeld, Lin \cite[pp.~226--227]{KL}. 
Note that every Dunford-Schwartz operator is also a contraction on  $L^p(X,\mu)$ for every $p\in(1,\infty)$, see, e.g., \cite[Theorem 8.23]{EFHN}.
The Jacobs-deLeeuw-Glicksberg decomposition is therefore valid for Dunford-Schwartz operators on $L^p(X,\mu)$ for every $p\in[1,\infty)$.

Let  $T$ be a Dunford-Schwartz operator on $(X,\mu)$. The \emph{(linear) modulus} $|T|$ of $T$ is defined as the unique positive operator on $L^1(X,\mu)$ having  the same $L^1$- and $L^\infty$-norm as $T$ such that $|T^nf|\leq |T|^n |f|$ holds a.e. for every $f\in L^1(X,\mu)$ and every $n\in\mb{N}$. The modulus of a Dunford-Schwartz operator is again a Dunford-Schwartz operator. For details, we refer to Dunford, Schwartz \cite[p.~672]{DS-book} and Krengel \cite[pp.~159--160]{K-book}. 
Also, it is easily seen that for $T$ Dunford-Schwartz, the operators $\lambda T$ ($\lambda\in\mb{T}$) are themselves Dunford-Schwartz and have the same modulus.

For example, every Koopman operator (i.e., an operator induced by a $\mu$-preserving transformation on $X$) is a positive Dunford-Schwartz operator, hence coincides with its modulus.

A key property of Dunford-Schwartz operators needed for the present paper is that the validity of pointwise ergodic theorems typically extends from Koopman operators to Dunford-Schwartz operators. 

For instance, for every $f\in L^1(X,\mu)$ the ergodic averages 
\begin{equation}\label{eq:pet}
\aveN T^nf
\end{equation}
converge a.e.~as $N\to\infty$, see Dunford, Schwartz \cite[p.~675]{DS-book}.

\smallskip

We shall also need to define some classes of sequences that act as good weights for pointwise ergodic theorems.

A sequence $(a_n)_{n\in\mb{N}}\subset\mb{C}$ is called a \emph{trigonometric polynomial} (cf. \cite{JLO}) if it is of the form $a_n=\sum_{j=1}^t b_j\rho_j^n$ where the $b_j$ are complex numbers, and $\rho_j\in\mb{T}$ for all $1\leq j\leq t$.\\
Let $\ms{P}\subset \ell^\infty$ denote the set of Bohr almost periodic sequences, i.e., the set of uniform limits of trigonometric polynomials.

The following properties of the set $\ms{P}$ will be used: It is closed in $l^\infty$, closed under multiplication, and
is a subclass of  (Weyl) almost periodic sequences $AP(\mb{N})$, i.e., sequences whose orbit under the left shift is relatively compact in $l^\infty$. Actually, $AP(\mb{N})=\ms{P}\oplus c_0$,  see Bellow, Losert \cite[p. 316]{BL}, corresponding to the Jacobs-deLeeuw-Glicksberg decomposition of $AP(\mb{N})$ induced by the left shift, see, e.g.,~\cite[Theorem I.1.20]{eisner-book}. 

By {\c{C}}{\"o}mez,  Lin, Olsen \cite[Theorem 2.5]{CLO}, every element $(a_n)_{n=1}^\infty$ of $AP(\mb{N})$, and hence of  $\ms{P}$, is a good weight for the pointwise ergodic theorem for Dunford-Schwartz operators. That is, for every Dunford-Schwartz operator $T$ on a probability space $(X,\mu)$ and every $f\in L^1(X,\mu)$, the weighted ergodic averages
$$
\aveN a_n T^nf
$$
converge almost everywhere as $N\to\infty$.

A sequence $(a_n)_{n\in\mb{N}}\subset\mb{C}$ is called \emph{linear} (cf. \cite{E}), if there exist a Banach space $E$, an operator $T\in \mc{L}(E)$ with relatively weakly compact orbits and $y\in E$, $y'\in E'$ such that $a_n=y'(T^ny)$ for all $n\in\mb{N}$. Let us call a linear sequence \emph{stable} if we can choose $y\in E_s$, and \emph{reversible} if we can chose $y\in E_r$. It is easy to see that stable linear sequences all lie in $\mc{N}$, whereas reversible linear sequences all lie in $\ms{P}$.

We shall later also need properties of polynomial subsequences of linear sequences, and thus a corresponding class of good weights for the pointwise polynomial ergodic theorem.

\begin{Def}
Given $1\leq p<\infty$ and a subsequence $(n_s)_{s\in\mb{N}}$ of $\mb{N}$, the class $B_{p,(n_s)_{s\in\mb{N}}}$ of $p,(n_s)_{s\in\mb{N}}$-Besicovitch sequences is the closure of the trigonometric polynomials in the ${p,(n_s)_{s\in\mb{N}}}$
semi-norm defined by
\[
\|(a_n)_{n\in\mb{N}}\|_{p,(n_s)_{s\in\mb{N}}}^p=\limsup_{N\to\infty} \frac{1}{N} \sum_{n=1}^N |a_{n_s}|^p.
\]
\end{Def}

By \cite[Thm. 2.1]{LO}, the set of bounded sequences in these classes is independent of the choice of $p$, i.e.,  $B_{1,(n_s)_{s\in\mb{N}}}\cap l^\infty=B_{p,(n_s)_{s\in\mb{N}}}\cap l^\infty $ for all $p\in(1,\infty)$. Note that the seminorm defined above is trivially dominated by the $l^\infty$ norm, and hence $\ms{P}\subset B_{1,(n_s)_{s\in\mb{N}}}\cap l^\infty$ for any subsequence $(n_s)_{s\in\mb{N}}$ of $\mb{N}$. The closedness of $\ms{P}$ under multiplication thus yields the following lemma.

\begin{Le}\label{Le:Besicovitch}
Let $(a_{n;j})_{n\in\mb{N}}$ be a reversible linear sequence for each $1\leq j\leq t$. Then $(b_n)_{n\in\mb{N}}$ defined by $b_n:=\prod_{j=1}^t a_{n;j}$ lies in $B_{1,(n_s)_{s\in\mb{N}}}\cap l^\infty$ for any subsequence $(n_s)_{s\in\mb{N}}$ of $\mb{N}$.
\end{Le}

The essential property of elements of $B_{1,(n_s)_{s\in\mb{N}}}\cap l^\infty$ is given by the following theorem.

\begin{Thm}\label{Thm:Besicovitch}(cf. \cite[Theorem 2.1]{JLO}) Let $T$ be a Dunford-Schwartz operator on a standard probability space $(X,\mu)$, $1\leq p<\infty$, and $\mf{q}(x)$ a polynomial with integer coefficients taking positive values on $\mb{N}$. Then for any $f\in L^p(X,\mu)$ and $(b_n)_{n\in\mb{N}}\in B_{1,(\mf{q}(n))_{n\in\mb{N}}}\cap l^\infty$ the limit
\[
\lim_{N\to\infty}\frac{1}{N} \sum_{n=1}^N b_{\mf{q}(n)}T^{\mf{q}(n)}f
\]
exists almost surely.
\end{Thm}

Finally, we need information about the sequences $(\lambda_{*;n})$ along polynomial indices. Recall that an operator is almost weakly stable if the stable part of the Jacobs-Glicksberg-deLeeuw decomposition is the whole space.

\begin{Prop}[cf. \cite{KK_AWPS} Thm. 1.1]\label{prop:uaws_pol}
Let $T$ be an almost weakly stable contraction on a Hilbert space $H$. Then $T$ is almost weakly polynomial stable, i.e., for any $h\in H$ and non-constant polynomial $\mf{q}$ with integer coefficients taking positive values on $\mb{N}$, the sequence $\{T^{\mf{q}(j)}h\}_{j=1}^\infty$ is almost weakly stable.
\end{Prop}

As a consequence we obtain the following result.

\begin{Cor}\label{cor:awps}
Let $T$ be a Dunford-Schwartz operator on the standard probability space $(X,\mu)$, $\mf{q}$ a non-constant polynomial with integer coefficients taking positive values on $\mb{N}$ and $A$ an arbitrary operator on $L^2(X,\mu)$.
Then for any $g,\varphi\in L^2(X,\mu)$ with $g$ in the stable part of $L^2(X,\mu)$ with respect to $T$, we have that the sequence $\langle AT^{\mf{q}(n)}g,\varphi\rangle$ is bounded and lies in $\mc{N}$.
\end{Cor}

%
%

\section{Proof of Theorem \ref{thm:main}}\label{sec:general-case}

We shall proceed by successive splitting and reduction. For each operator $T_i$, starting from $T_2$, we split the functions it is applied to into several terms using condition (A1). Most of the obtained terms can be easily dealt with, but for the remaining ''difficult'' terms, we move on to $T_{i+1}$, up to and including $T_m$.\\
We first prove part (1), and then use this result to complete the proof for part (2). 

In what follows, we shall assume without loss of generality that for the constant in Theorem \ref{thm:main}, we have $C\geq 1$.
Given a function $f\in E_{1;s}$, and an $\varepsilon\in(0,1)$, do the following.
\begin{enumerate}[(I)]
\item First, set $d=0$, $c:=\varepsilon C^{-m}$ and let $\mc{I}_0$ consist of the empty index.
\item By assumption (A1), for each $f_{v}$ ($v\in\mc{I}_d$) we may find a decomposition $E=\mc{U}_v\oplus \mc{R}_v$ with $\ell_v:=\dim\mc{U}_v<\infty$ and 
\[
P_{\mc{R}_v}\ms{A}_{d+1,f_{v}}\subset B_{c_v}(0,L^\infty(X,\mu)).
\]
For each $v\in\mc{I}_d$, choose a maximal linearly independent set $f_{v;1},\ldots,f_{v;\ell_v}$ in $\mc{U}_v$. We can then for each $n\in\mb{N}$ write the unique decomposition
\[
A_{d+1}T_{d+1}^{n}f_{v}=\lambda_{v,1;n}f_{v,1}+\ldots+\lambda_{v,\ell_v;n}f_{v,\ell_v}+r_{v;n},
\]
for appropriate coefficients $\lambda_{v,j;n}\in\mb{C}$ and some remainder term $r_{v,n}\in \mc{R}_v$ with $\|r_{v;n}\|_\infty<c_v$.
Choose further elements $\varphi_{v;1},\ldots, \varphi_{v;\ell_v}\in E'$ with the property
\[
\varphi_{v;i}(f_{v,j})=\delta_{i,j}\quad \mbox{and}\quad \varphi_{v;i}|_{\mc{R}_v}=0\quad  \mbox{ for every } i,j\in\{1,\ldots,\ell_v\}.
\]
Set
\[
u_v:=\|f_v\|\cdot\|A_{d+1}^*\|\max_{1\leq j\leq \ell_v} \|\varphi_{v;j}\|.
\]
\item
Let
\[
\mc{I}_{d+1}:=\left\{w\in\mb{N}^{d+1}|\overline{w}\in\mc{I}_d, 1\leq w^*\leq \ell_{\overline{w}}\right\}.
\]
Also, for each $w\in\mc{I}^{d+1}$, let $c_w:=c_{\overline w}/u_{\overline{w}}\ell_{\overline w}$.
\item Increase $d$ by 1, and unless $d=m-1$, start anew from step (II).
\item For each $w\in \mc{I}_{m-1}$, choose the function $\widetilde{f}_w\in L^\infty$ such that
\[
\|f_w-\widetilde{f}_w\|_1\leq\|f_w-\widetilde{f}_w\|_p<
c_w\cdot \varepsilon/|\mathcal{I}_{m-1}|.
\]
\end{enumerate}

\noindent\textbf{Proof of (1)}.\\
Applying the above splitting procedure to $f\in E_{1;s}$, we may bound our original Cesàro averages by a finite sum of averages. For a.e. $z\in X$ we have
\begin{align*}
&\frac{1}{N^k}\sum_{1\leq n_1,\ldots, n_k\leq N} \left|T_m^{n_{\alpha(m)}}A_{m-1}T^{n_{\alpha(m-1)}}_{m-1}\ldots A_2T_2^{n_{\alpha(2)}}A_1T_1^{n_{\alpha(1)}} f \right|(z)\\
\leq&
\sum_{v\in \mc{I}_{m-1}}
\frac{1}{N^k}\sum_{1\leq n_1,\ldots, n_k\leq N}
\left|
T_m^{n_{\alpha(m)}}f_v
\right|(z)
\prod_{l(x)>0,\, x\subseteq v} \left|\lambda_{x;n_{\alpha(l(x))}} \right|\\
&+\sum_{w\in\mc{I}_{l(w)},\,0\leq l(w)<m-1}
\frac{C^{m-2}}{N^k}\sum_{1\leq n_1,\ldots, n_k\leq N}
\left|r_{w;n_{\alpha(l(w)+1)}}(z) \right|
\prod_{l(x)>0,\,x\subseteq w} \left|\lambda_{x;n_{\alpha(l(x))}} \right|\\
\leq&
\sum_{v\in \mc{I}_{m-1}}
\frac{1}{N^k}\sum_{1\leq n_1,\ldots, n_k\leq N}
\left|
T_m^{n_{\alpha(m)}}\widetilde{f}_v
\right|(z)
\prod_{l(x)>0,\, x\subseteq v} \left|\lambda_{x;n_{\alpha(l(x))}} \right|\\
&+\sum_{v\in \mc{I}_{m-1}}
\frac{1}{N^k}\sum_{1\leq n_1,\ldots, n_k\leq N}
\left|T_m^{n_{\alpha(m)}}\left(
f_v-\widetilde{f}_v
\right)
\right|(z)
\prod_{l(x)>0,\, x\subseteq v} \left|\lambda_{x;n_{\alpha(l(x))}} \right|\\
&+\sum_{w\in\mc{I}_{l(w)},\,0\leq l(w)<m-1}
\frac{C^{m}}{N^k}\sum_{1\leq n_1,\ldots, n_k\leq N}
c_w
\prod_{l(x)>0,\,x\subseteq w} \left|\lambda_{x;n_{\alpha(l(x))}} \right|.
\end{align*}

We shall bound each of these three sums separately. Note that by the definition of the linear forms, we have for each $v\in\mc{I}_{l(v)}$ ($1\leq l(v)\leq m-1$)
\[
\lambda_{v;n}=\varphi_{\overline{v};v*}(A_{l(v)}T_{l(v)}^nf_{\overline{v}})=(A_{l(v)}^*\varphi_{\overline{v};v*})(T_{l(v)}^nf_{\overline{v}}),
\]
and hence
\[
\left|\lambda_{v;n}\right|\leq \|f_{\overline{v}}\|\cdot\|A_{l(v)}^*\|\max_{1\leq j\leq \ell_v} \|\varphi_{\overline{v};j}\|=
u_{\overline{v}},
\]
but also, since $f\in E_{1;s}$, we have $(\lambda_{j;n})_{n\in\mb{N}}\in\mc{N}$ for each $1\leq j\leq \ell$.

Using that $\mc{N}$ is closed under multiplication by bounded sequences, on the one hand we obtain that
\begin{align*}
&\lim_{N\to\infty} \sum_{w\in\mc{I}_{l(w)},\,0\leq l(w)<m-1}
\frac{C^{m}}{N^k}\sum_{1\leq n_1,\ldots, n_k\leq N}
c_w
\prod_{l(x)>0,\,x\subseteq w} \left|\lambda_{x;n_{\alpha(l(x))}} \right|\\
=&\sum_{w\in\mc{I}_{l(w)},\,0\leq l(w)<m-1}
C^{m}c_w
\left(
\lim_{N\to\infty} 
\frac{1}{N^k}\sum_{1\leq n_1,\ldots, n_k\leq N}
\prod_{l(x)>0,\,x\subseteq w} \left|\lambda_{x;n_{\alpha(l(x))}} \right|
\right)
\\
=&C^{m}\sum_{w\in\mc{I}_{0}} c_w=C^mc=\varepsilon.
\end{align*}
On the other hand, also using that $\widetilde{f}_v$ is essentially bounded for each $v\in\mc{I}_{m-1}$ and that $T_m$ as a Dunford-Schwartz operator is a contraction on $L^\infty$, we obtain that
\begin{align*}
\lim_{N\to\infty}\sum_{v\in \mc{I}_{m-1}}
\frac{1}{N^k}\sum_{1\leq n_1,\ldots, n_k\leq N}
\left|
T_m^{n_{\alpha(m)}}\widetilde{f}_v
\right|(z)
\prod_{l(x)>0,\, x\subseteq v} \left|\lambda_{x;n_{\alpha(l(x))}} \right|=0
\end{align*}
for almost every $z\in X$.

Thus only the middle sum remains to be bounded. To treat that term, we shall make use of the Pointwise Ergodic Theorem for Dunford-Schwartz operators. Since the modulus of a DS operator is itself DS, we may apply the PET to $|T_m|$ and the functions $|f_v-\widetilde{f}_v|$ to obtain that for each $v\in\mc{I}_{m-1}$, there exists a function $0\leq\mf{f}_v\in L^1$ with $\|\mf{f}_v\|_1\leq |f_v-\widetilde{f}_v|_1$ and a set $S_v$ with $\mu(S_v)=1$ such that
\[
\lim_{N\to\infty}\frac{1}{N} \sum_{n=1}^N |T_m|^{n} |(f_v-\widetilde{f}_v)|(z)=\mf{f}_v(z)
\]
for all $z\in S_v$.
Note that by the norm bound in step (V), there then exists a set $\mf{S}_v\subset S_v$ with $\mu(\mf{S}_v)>1-\varepsilon/|\mc{I}_{m-1}|$ such that 
\[
\lim_{N\to\infty}\frac{1}{N} \sum_{n=1}^N |T_m|^{n} |(f_v-\widetilde{f}_v)|(z)\leq c_w
\]
for all $z\in\mf{S}_v$.
We obtain that for every $z\in \bigcap_{v\in\mc{I}_{m-1}}\mf{S}_v$ we have
\begin{align*}
&\limsup_{N\to\infty}\sum_{v\in \mc{I}_{m-1}}
\frac{1}{N^k}\sum_{1\leq n_1,\ldots, n_k\leq N}
\left|T_m^{n_{\alpha(m)}}\left(
f_v-\widetilde{f}_v
\right)
\right|(z)
\prod_{l(x)>0,\, x\subseteq v} \left|\lambda_{x;n_{\alpha(l(x))}} \right|\\
\leq&
\lim_{N\to\infty}\sum_{v\in \mc{I}_{m-1}}
\frac{1}{N^k}\sum_{1\leq n_1,\ldots, n_k\leq N}
\left(\left|T_m\right|^{n_{\alpha(m)}}\left|
f_v-\widetilde{f}_v
\right|\right)(z)
\prod_{l(x)>0,\, x\subseteq v} u_{\overline{x}}\\
\leq&\sum_{v\in \mc{I}_{m-1}}
c_{v}
\prod_{l(x)>0,\, x\subseteq v} u_{\overline{x}}=\varepsilon C^{-m}\leq\varepsilon.
\end{align*}

In total, we obtain that for every $\varepsilon>0$, we have
\[
\limsup_{N\to\infty}\frac{1}{N^k}\sum_{1\leq n_1,\ldots, n_k\leq N} \left|T_m^{n_{\alpha(m)}}A_{m-1}T^{n_{\alpha(m-1)}}_{m-1}\ldots A_2T_2^{n_{\alpha(2)}}A_1T_1^{n_{\alpha(1)}} f \right|(z)<\varepsilon
\]
for all $z\in \bigcap_{v\in\mc{I}_{m-1}}\mf{S}_v$. Since $\mu(\bigcap_{v\in\mc{I}_{m-1}}\mf{S}_v)>1-|\mc{I}_{m-1}|\cdot\varepsilon/|\mc{I}_{m-1}|=1-\varepsilon$, letting $\varepsilon\to0$ concludes our proof of Part (1).

We now turn our attention to Part (2), and show a.e. convergence of the averages also on the reversible part $E_{1;r}$ with respect to the operator $T_1$. Again we shall proceed by iterated splitting of the function, but part (1) will also be made use of.

Given a function $f\in E_{1;r}$, and an $\varepsilon>0$, do the following.
\begin{enumerate}[(i)]
\item First, set $d=0$, $c:=\varepsilon C^{-m}$ and let $\mc{I}_0$ consist of the empty index.
\item By assumption (A1), for each $f_{v}$ ($v\in\mc{I}_d$) we may find a decomposition $E=\mc{U}_v\oplus \mc{R}_v$ with $\ell_v:=\dim\mc{U}_v<\infty$ and 
\[
P_{\mc{R}_v}\ms{A}_{d+1,f_{v}}\subset B_{c_v}(0,L^\infty(X,\mu)).
\]
For each $v\in\mc{I}_d$, choose a maximal linearly independent set $g_{v,1},\ldots,g_{v,\ell_v}$ in $\mc{U}_v$. We can then for each $n\in\mb{N}$ write the unique decomposition
\[
A_{d+1}T_{d+1}^{n}f_{v}=\lambda_{v,1;n}g_{v,1}+\ldots+\lambda_{v,\ell_v;n}g_{v,\ell_v}+r_{v;n},
\]
for appropriate coefficients $\lambda_{v,j;n}\in\mb{C}$ and some remainder term $r_{v,n}\in \mc{R}_v$ with $\|r_{v;n}\|_\infty<c_v$.
Choose further elements $\varphi_{v;1},\ldots, \varphi_{v;\ell_v}\in E'$ with the property
\[
\varphi_{v;i}(g_{v,j})=\delta_{i,j}\quad \mbox{and}\quad \varphi_{v;i}|_{\mc{R}_v}=0\quad  \mbox{ for every } i,j\in\{1,\ldots,\ell_v\}.
\]
Set
\[
u_v:=\|f_v\|\cdot\|A_{d+1}^*\|\max_{1\leq j\leq \ell_v} \|\varphi_{v;j}\|.
\]
\item For each $v\in\mc{I}_d$ and $1\leq j\leq\ell_v$, let $f_{v,j}:=P_{E_{d+2;r}}g_{v,j}$ be the reversible part of $g_{v,j}$ with respect to $T_{d+2}$, and let $q_{v,j}:=g_{v,j}-f_{v,j}$ be its stable part.
\item
Let
\[
\mc{I}_{d+1}:=\left\{w\in\mb{N}^{d+1}|\overline{w}\in\mc{I}_d, 1\leq w^*\leq \ell_{\overline{w}}\right\}.
\]
Also, for each $w\in\mc{I}^{d+1}$, let $c_w:=c_{\overline w}/u_{\overline{w}}\ell_{\overline w}$.
\item Increase $d$ by 1, and unless $d=m-1$, start anew from step (II).
\end{enumerate}

\noindent\textbf{Proof of (2)}

Let us apply the iterated decomposition (i)--(vi) detailed above to the function $f\in E_{1;r}$. We obtain that
\begin{align*}
&\frac{1}{N^k}\sum_{1\leq n_1,\ldots, n_k\leq N} T_m^{n_{\alpha(m)}}A_{m-1}T^{n_{\alpha(m-1)}}_{m-1}\ldots A_2T_2^{n_{\alpha(2)}}A_1T_1^{n_{\alpha(1)}} f\\
=&
\sum_{v\in \mc{I}_{m-1}}
\frac{1}{N^k}\sum_{1\leq n_1,\ldots, n_k\leq N}
\left(\prod_{l(x)>0,\, x\subseteq v} \lambda_{x;n_{\alpha(l(x))}}\right)
T_m^{n_{\alpha(m)}}g_v
\\
+&\sum_{w\in\mc{I}_{l(w)},\,0< l(w)<m-1}
\frac{1}{N^k}\sum_{1\leq n_1,\ldots, n_k\leq N}\\
& T_m^{n_{\alpha(m)}}A_{m-1}T^{n_{\alpha(m-1)}}_{m-1}\ldots 
 A_{\ell(w)+1}T_{\ell(w)+1}^{n_{\alpha(\ell(w)+1)}}
 q_{w} \prod_{l(x)>0,\,x\subseteq w} \lambda_{x;n_{\alpha(l(x))}}
\\
+&\sum_{w\in\mc{I}_{l(w)},\,0\leq l(w)<m-1}
\frac{1}{N^k}\sum_{1\leq n_1,\ldots, n_k\leq N}\\
& T_m^{n_{\alpha(m)}}A_{m-1}T^{n_{\alpha(m-1)}}_{m-1}\ldots 
 A_{\ell(w)+2}T_{\ell(w)+2}^{n_{\alpha(\ell(w)+2)}}
 r_{w;n_{\alpha(l(w)+1)}} \prod_{l(x)>0,\,x\subseteq w} \lambda_{x;n_{\alpha(l(x))}}
.
\end{align*}

First, let us look at the terms involving the $q_w$-s. For each $w\in \mc{I}_{l(w)}$ with $0<l(w)<m-1$, we note that the products $\prod_{l(x)>0,\,x\subseteq w} \lambda_{x;n_{\alpha(l(x))}}$ are bounded in absolute value by the constant $\prod_{l(x)>0,\,x\subseteq w} u_{\overline{x}}$, and using part (1) with the new value $m':=m-l(w)>1$, we obtain for each $w$ that
\begin{align*}
&\limsup_{N\to\infty} \left|
\frac{1}{N^k}\sum_{1\leq n_1,\ldots, n_k\leq N}T_m^{n_{\alpha(m)}}A_{m-1}T^{n_{\alpha(m-1)}}_{m-1}\ldots 
 A_{\ell(w)+1}T_{\ell(w)+1}^{n_{\alpha(\ell(w)+1)}}
 q_{w} \prod_{l(x)>0,\,x\subseteq w} \lambda_{x;n_{\alpha(l(x))}}
\right|(z)\\
\leq&\left(\prod_{l(x)>0,\,x\subseteq w} u_{\overline{x}}\right)\lim_{N\to\infty} \frac{1}{N^k}\sum_{1\leq n_1,\ldots, n_k\leq N}\left|
T_m^{n_{\alpha(m)}}A_{m-1}T^{n_{\alpha(m-1)}}_{m-1}\ldots 
 A_{\ell(w)+1}T_{\ell(w)+1}^{n_{\alpha(\ell(w)+1)}}
 q_{w}\right| (z)=0
\end{align*}
for almost every $z\in X$. Since there are finitely many different $q_w$ terms, they contribute a total of 0 to the Cesàro means on a set of full measure.

Second, let us look at the terms involving the $r_{w;*}$-s. Note that since we work on the reversible part and lack a coefficient sequence $\lambda_*$ in $\mc{N}$, we cannot conclude the same way as in part (1). Let us therefore fix $w\in\mc{I}_{l(w)}$ with $0\leq l(w)<m-1$. We then have using (A2) that
\begin{align*}
&\left\| T_m^{n_{\alpha(m)}}A_{m-1}T^{n_{\alpha(m-1)}}_{m-1}\ldots 
 A_{\ell(w)+2}T_{\ell(w)+2}^{n_{\alpha(\ell(w)+2)}}
 r_{w;n_{\alpha(l(w)+1)}} \prod_{l(x)>0,\,x\subseteq w} \lambda_{x;n_{\alpha(l(x))}}\right\|_\infty\\
 \leq&C^{m-l(w)-2}\left\|r_{w;n_{\alpha(l(w)+1)}}\right\|\prod_{l(x)>0,\,x\subseteq w} u_{\overline{x}}<C^mc_w\prod_{l(x)>0,\,x\subseteq w} u_{\overline{x}}=\varepsilon\prod_{x\subset w} \frac{1}{\ell_{x}}.
\end{align*}
This in turn implies that for every $N$
\begin{align*}
&\left\|
\sum_{w\in\mc{I}_{l(w)},\,0\leq l(w)<m-1}
\frac{1}{N^k}\sum_{1\leq n_1,\ldots, n_k\leq N}
\right.\\
& \left.T_m^{n_{\alpha(m)}}A_{m-1}T^{n_{\alpha(m-1)}}_{m-1}\ldots 
 A_{\ell(w)+2}T_{\ell(w)+2}^{n_{\alpha(\ell(w)+2)}}
 r_{w;n_{\alpha(l(w)+1)}} \prod_{l(x)>0,\,x\subseteq w} \lambda_{x;n_{\alpha(l(x))}}
\right\|_\infty\\
\leq&
\sum_{w\in\mc{I}_{l(w)},\,0\leq l(w)<m-1}
\frac{1}{N^k}\sum_{1\leq n_1,\ldots, n_k\leq N}\\
&\left\| T_m^{n_{\alpha(m)}}A_{m-1}T^{n_{\alpha(m-1)}}_{m-1}\ldots 
 A_{\ell(w)+2}T_{\ell(w)+2}^{n_{\alpha(\ell(w)+2)}}
 r_{w;n_{\alpha(l(w)+1)}} \prod_{l(x)>0,\,x\subseteq w} \lambda_{x;n_{\alpha(l(x))}}
\right\|_\infty\\
<&\sum_{w\in\mc{I}_{l(w)},\,0\leq l(w)<m-1}\varepsilon\prod_{x\subset w} \frac{1}{\ell_{x}}=
\varepsilon\sum_{d=1}^{m-2}\sum_{w\in\mc{I}_d}\prod_{x\subset w} \frac{1}{\ell_{x}}=\varepsilon\sum_{d=1}^{m-2}1=\varepsilon(m-2).
\end{align*}

It only remains to estimate the terms involving the functions $g_v$ ($v\in\mc{I}_{m-1}$). We have
\begin{align*}
&\frac{1}{N^k}\sum_{1\leq n_1,\ldots, n_k\leq N}
\left(\prod_{l(x)>0,\, x\subseteq v} \lambda_{x;n_{\alpha(l(x))}}\right)
T_m^{n_{\alpha(m)}}g_v\\
=&\left(\frac{1}{N^{k-1}}
\sum_{1\leq n_j\leq N \,(1\leq j\leq k,\, j\neq\alpha(m))}
\left(\prod_{l(x)>0,\,\alpha(l(x))\neq\alpha(m),\, x\subseteq v} \lambda_{x;n_{\alpha(l(x))}}\right)
\right)\\
&\cdot\left(
\frac{1}{N}\sum_{n=1}^N 
\left(\prod_{l(x)>0,\,\alpha(l(x))=\alpha(m),\, x\subseteq v} \lambda_{x;n}\right)T_m^n g_v
\right).
\end{align*}

We shall show that as $N$ tends to infinity, the first, complex valued factor is convergent, whereas the second, function valued factor converges almost everywhere. This will then imply that the product also converges almost everywhere.

Let us fix $v\in\mc{I}_{m-1}$.
We obtain for each $x\subseteq v$ with $l(x)>0$ that
\begin{align*}
&\lambda_{x;n}=\varphi_{\overline{x};x*}\left(A_{l(x)}T_{l(x)}^nf_{\overline{x}}\right)=\langle A_{l(x)}^*\varphi_{\overline{x};x*},T_{l(x)}^nf_{\overline{x}}\rangle
\end{align*}
and since $f_{\overline{x}}$ is in the reversible part of $E$ with respect to $T_{l(x)}$, the sequence $(\lambda_{x;n})_{n\in\mb{N}}$ is a reversible linear sequence.
Using that $\ms{P}$ is closed under multiplication, we have that for each $1\leq j\leq m$
\[
\left(\prod_{l(x)>0,\,\alpha(l(x))=j,\, x\subseteq v} \lambda_{x;n}\right)_{n\in\mb{N}}\in \ms{P}.
\]
In particular, for each $v\in\mc{I}_ {m-1}$, the Cesàro means
\[
\left(\frac{1}{N^{k-1}}
\sum_{1\leq n_j\leq N \,(1\leq j\leq k,\, j\neq\alpha(m))}
\left(\prod_{l(x)>0,\,\alpha(l(x))\neq\alpha(m),\, x\subseteq v} \lambda_{x;n_{\alpha(l(x))}}\right)
\right)
\]
converge.

Finally, let us turn our attention to the factor
\[
\frac{1}{N}\sum_{n=1}^N 
\left(\prod_{l(x)>0,\,\alpha(l(x))=\alpha(m),\, x\subseteq v} \lambda_{x;n}\right)T_m^n g_v.
\]
Since elements of $\ms{P}$ are good weights for the PET for Dunford-Schwartz operators, this converges poinwise almost everywhere.

In conclusion, for almost every $z\in X$ we have
\begin{align*}
&(\limsup_{N\to\infty}-\liminf_{N\to\infty})\frac{1}{N^k}\sum_{1\leq n_1,\ldots, n_k\leq N}
\left(T_m^{n_{\alpha(m)}}A_{m-1}T^{n_{\alpha(m-1)}}_{m-1}\ldots A_2T_2^{n_{\alpha(2)}}A_1T_1^{n_{\alpha(1)}} f\right)(z)\\
\leq&
\sum_{v\in \mc{I}_{m-1}}
(\limsup_{N\to\infty}-\liminf_{N\to\infty})\frac{1}{N^k}\sum_{1\leq n_1,\ldots, n_k\leq N}
\left(\prod_{l(x)>0,\, x\subseteq v} \lambda_{x;n_{\alpha(l(x))}}\right)
\left(T_m^{n_{\alpha(m)}}g_v\right)(z)
\\
+&\sum_{w\in\mc{I}_{l(w)},\,0< l(w)<m-1}
(\limsup_{N\to\infty}-\liminf_{N\to\infty})\frac{1}{N^k}\sum_{1\leq n_1,\ldots, n_k\leq N}\\
& \left(T_m^{n_{\alpha(m)}}A_{m-1}T^{n_{\alpha(m-1)}}_{m-1}\ldots 
 A_{\ell(w)+1}T_{\ell(w)+1}^{n_{\alpha(\ell(w)+1)}}
 q_{w}\right)(z) \prod_{l(x)>0,\,x\subseteq w} \lambda_{x;n_{\alpha(l(x))}}
\\
+&\sum_{w\in\mc{I}_{l(w)},\,0\leq l(w)<m-1}
(\limsup_{N\to\infty}-\liminf_{N\to\infty})\frac{1}{N^k}\sum_{1\leq n_1,\ldots, n_k\leq N}\\
& \left(T_m^{n_{\alpha(m)}}A_{m-1}T^{n_{\alpha(m-1)}}_{m-1}\ldots 
 A_{\ell(w)+2}T_{\ell(w)+2}^{n_{\alpha(\ell(w)+2)}}
 r_{w;n_{\alpha(l(w)+1)}}\right)(z) \prod_{l(x)>0,\,x\subseteq w} \lambda_{x;n_{\alpha(l(x))}}\\
 =&0+0+\sum_{w\in\mc{I}_{l(w)},\,0\leq l(w)<m-1}
(\limsup_{N\to\infty}-\liminf_{N\to\infty})\frac{1}{N^k}\sum_{1\leq n_1,\ldots, n_k\leq N}\\
& \left(T_m^{n_{\alpha(m)}}A_{m-1}T^{n_{\alpha(m-1)}}_{m-1}\ldots 
 A_{\ell(w)+2}T_{\ell(w)+2}^{n_{\alpha(\ell(w)+2)}}
 r_{w;n_{\alpha(l(w)+1)}}\right)(z) \prod_{l(x)>0,\,x\subseteq w} \lambda_{x;n_{\alpha(l(x))}}\\
\leq&
2\sup_{N\in\mb{N}}\left\|
\sum_{w\in\mc{I}_{l(w)},\,0\leq l(w)<m-1}
\frac{1}{N^k}\sum_{1\leq n_1,\ldots, n_k\leq N}
\right.\\
& \left.T_m^{n_{\alpha(m)}}A_{m-1}T^{n_{\alpha(m-1)}}_{m-1}\ldots 
 A_{\ell(w)+2}T_{\ell(w)+2}^{n_{\alpha(\ell(w)+2)}}
 r_{w;n_{\alpha(l(w)+1)}} \prod_{l(x)>0,\,x\subseteq w} \lambda_{x;n_{\alpha(l(x))}}
\right\|_\infty \\
\leq& 2\varepsilon(m-2).
\end{align*}

Since this holds for every $\varepsilon>0$, this concludes the proof of part (2).

\begin{Rem}
The pointwise limit is -- if it exists -- clearly the same as the stong limit, and takes the form given in \cite[Thm. 3]{EKK}.
\end{Rem}

%
%
\section{Pointwise polynomial ergodic version}

In this section our goal is to prove a polynomial version of Theorem \ref{thm:main}.

With these tools in hand, we can now state and prove almost everywhere pointwise convergence of entangled means on Hilbert spaces.

\begin{Thm}\label{thm:main_poly}
Let $m>1$ and $k$ be positive integers, $\alpha:\left\{1,\ldots,m\right\}\to\left\{1,\ldots,k\right\}$ a not necessarily surjective map, and $T_1,T_2,\ldots, T_m$ Dunford-Schwartz operators on a standard probability space $(X,\mu)$.
Let $E:=L^2(X,\mu)$ and let $E=E_{j,r}\oplus E_{j,s}$ be the Jacobs-Glicksberg-deLeeuw decomposition corresponding to $T_j$ $(1\leq j\leq m)$. Let further $A_j\in\mc{L}(E)$ $(1\leq j< m)$ be bounded operators. Suppose that the conditions (A1) and (A2) of Theorem \ref{thm:main} hold.\\
Further, let $\mf{q}_1,\mf{q}_2,\ldots, \mf{q}_k$ be non-constant polynomials with integer coefficients taking positive values on $\mb{N}$.
Then we have the following:
\begin{enumerate}
\item for each $f\in E_{1,s}$, 
\[
\frac{1}{N^k}\sum_{1\leq n_1,\ldots, n_k\leq N} \left|T_m^{\mf{q}_{\alpha(m)}(n_{\alpha(m)})}
\ldots A_2T_2^{\mf{q}_{\alpha(2)}(n_{\alpha(2)})}A_1T_1^{\mf{q}_{\alpha(1)}(n_{\alpha(1)})} f \right|\rightarrow 0
\]
pointwise a.e.;
\item 
for each $f\in E_{1,r}$, the averages 
\[
\frac{1}{N^k}\sum_{1\leq n_1,\ldots, n_k\leq N} T_m^{\mf{q}_{\alpha(m)}(n_{\alpha(m)})}
\ldots A_2T_2^{\mf{q}_{\alpha(2)}(n_{\alpha(2)})}A_1T_1^{\mf{q}_{\alpha(1)}(n_{\alpha(1)})} f
\]
converge pointwise almost everywhere.
\end{enumerate}
\end{Thm}

\begin{proof}
We shall follow the proof of Theorem \ref{thm:main}, using the same recursive splitting. The question is then why the convergences still hold when averaging along polynomial subsequences.\\
For part (1), we have three terms to bound: those involving the remainder functions $r_{*;n}$, the ones involving the essentially bounded functions $\widetilde{f}_*$, and finally the ones with the small approximation errors $f_*-\widetilde{f}_*$. Using Corollary \ref{cor:awps}, we obtain that the subsequences $\lambda_{j;q(n)}$ involved ($1\leq j\leq\ell$) also lie in $\mc{N}$, leading to the same bounds as in the linear case for the first two types of terms. For the terms involving the functions $f_*-\widetilde{f}_*$, we use the polynomial version of PET for Dunford-Schwartz operators, Theorem \ref{Thm:Besicovitch}, to obtain that
for each $v\in\mc{I}_{m-1}$, there exists a function $0\leq\mf{f}_v\in L^1$ and a set $S_v$ with $\mu(S_v)=1$ such that
\[
\lim_{N\to\infty}\frac{1}{N} \sum_{n=1}^N |T_m|^{\mf{q}(n)} |(f_v-\widetilde{f}_v)|(z)=\mf{f}_{v}(z).
\]
for all $z\in S_v$. Since the polynomial Cesàro means are also contractive in $L^1$ for Dunford-Schwartz operators, the rest of the arguments remain unchanged, and this concludes the proof of part (1).

For part (2), we again have three types of terms. The terms involving the functions $q_*$ can again be treated using part (1) and shown to have a zero contribution almost everywhere, and the terms with the $r_{*;n}$-s also do not require any change in the arguments used. Only the terms involving the functions $g_v$ ($v\in\mc{I}_{m-1}$) remain. For these, we use Lemma \ref{Le:Besicovitch} combined with Theorem \ref{Thm:Besicovitch} to obtain the almost everywhere convergence needed. 

\end{proof}

%
%
%
%

\section{The continuous case}\label{sec:ex}

In this section, we finally turn our attention to a variant of the above results, where we replace the discrete action of the Dunford-Schwartz operators with the continuous action $C_0$-semigroups. In other words, the semigroups $\{T_i^n|n\in\mb{N}^+\}$ are replaced by strongly continuous semigroups $\{T_i(t)|t\in[0,\infty)\}$.\\

Let $T(\cdot):=(T(t))_{t\in [0,\infty)}$ be a $C_0$-semigroup of Dunford-Schwartz operators on $L^1(X,\mu)$. Then, by the standard approximation argument, using that the unit ball in $L^\infty(X,\mu)$ is invariant under the semigroup, $T(\cdot)$ is  automatically  a $C_0$-semigroup  (of contractions) on $L^p(X,\mu)$ for every $\infty>p\geq 1$. In addition, by Fubini's theo\-rem, see, e.g., Sato \cite[p.~3]{S}, for every $f\in L^1(X,\mu)$ the function $(T(\cdot)f)(x)$ is Lebesgue integrable over finite intervals in $[0,\infty)$ for almost every $x\in X$. Similarly, for $C_0$-semigroups $T_0(\cdot),\ldots,T_a(\cdot)$ on $E:=L^p(X,\mu)$, operators $A_0,\ldots,A_{a-1}\in\mc{L}(E)$ and $f\in E$, the product 
$$
(T_a(\cdot)A_{a-1}T_{a-1}(\cdot)\ldots A_1T_1(\cdot)A_0T_0(\cdot) f)(x)
$$
is  Lebesgue integrable over finite intervals in $[0,\infty)$ for almost every $x\in X$.

By Dunford, Schwartz \cite[pp. 694, 708]{DS-book}, the pointwise ergodic theorem extends to every strongly measurable semigroup $T(\cdot)$ of Dunford-Schwartz operators.
In addition, it can be shown through a simple adaptation of the arguments in Lin, Olsen, Tempelman \cite[Proof of Prop. 2.6]{LOT} that every $C_0$-semigroup of Dunford-Schwartz operators has relatively weakly compact orbits in $L^1(X,\mu)$. Thus, the continuous version of the Jacobs-deLeeuw-Glicksberg decomposition (see e.g.~\cite[Theorem III.5.7]{eisner-book}) is valid for such semigroups.

In the discrete case, the modulus $|T|$ of the operator $T$ was used to obtain a discrete semigroup of positive operators that dominates $(T^n)_{n\in\mb{N}}$ whilst keeping the Dunford-Schwartz property. The time-continuous case turns out to be more involved, as there is no ``first'' operator whose modulus can be used to generate the dominating semigroup. Just as in the discrete case, we usually have $|T^2|\neq|T|^2$, in the $C_0$ setting $(|T(t)|)_{t\geq 0}$ will generally not be a strongly continuous semigroup. By e.g.~Kipnis \cite{Ki} or Kubokawa \cite{Ku}, for a  $C_0$-semigroup $T(\cdot)$ of contractions there exists a minimal $C_0$-semigroup of positive operators dominating $T(\cdot)$, which we shall denote by $|T|(\cdot)$. Of course, $|T|(\cdot)=T(\cdot)$ for positive semigroups.  Moreover, the construction in \cite[pp. 372-3]{Ki} implies that if $T(\cdot)$ consists of Dunford-Schwartz operators then so does $|T|(\cdot)$.

With the above, the proof of Theorem \ref{thm:main} can be extended to the time-continuous setting to obtain the following $C_0$ version of our main theorem.

\begin{Thm}\label{thm:main-cont}
Let $m>1$ and $k$ be positive integers, $\alpha:\left\{1,\ldots,m\right\}\to\left\{1,\ldots,k\right\}$ a not necessarily surjective map and let 
 $(T_1(t))_{t\geq 0}$,$\ldots $,$(T_m(t))_{t\geq 0}$ be $C_0$-semigroups of Dunford-Schwartz operators on a standard probability space $(X,\mu)$.
Let $p\in[1,\infty)$, $E:=L^p(X,\mu)$ and let $E=E_{j,r}\oplus E_{j,s}$ be the Jacobs-Glicksberg-deLeeuw decomposition corresponding to $T_j(\cdot)$ $(1\leq j\leq m)$. Let further $A_j\in\mc{L}(E)$ $(1\leq j< m-1)$ be bounded operators. For a function $f\in E$ and an index $1\leq j\leq m-1$, write $\ms{A}_{j,f}:=\left\{A_jT_j(t)f\left|\right.t\in[0,\infty)\right\}$. Suppose that the following conditions hold:
\begin{itemize}
\item[(A1c)](Twisted compactness)
For any function $f\in E$, index $1\leq j\leq m-1$ and $\varepsilon>0$, there exists a decomposition $E=\mc{U}\oplus \mc{R}$ with $\dim \mc{U}<\infty$
such that
\[
P_{\mc{R}}\ms{A}_{j,f}
\subset B_\varepsilon(0,L^\infty(X,\mu)),
\]
with $P_{\mc{R}}$ denoting the projection onto $\mc{R}$ along $\mc{U}$.
\item[(A2c)](Joint $\mc{L}^\infty$-boundedness)
There exists a constant $C>0$ such that we have
\[\{A_jT_j(t)|\,t\in[0,\infty),1\leq j\leq m-1\}\subset B_C(0,\mc{L}(L^\infty(X,\mu)).
\]
\end{itemize}
Then we have the following:
\begin{enumerate}
\item for each $f\in E_{1,s}$, 
\[
\lim_{\mc{T}\to\infty}\frac{1}{\mc{T}^k}\int_{\left\{t_1,\ldots, t_k\right\}\in [0,\mc{T}]^k} \left|T_m(t_{\alpha(m)})
\ldots A_2T_2(t_{\alpha(2)})A_1T_1(t_{\alpha(1)}) f \right|\rightarrow 0
\]
pointwise a.e.;
\item
for each $f\in E_{1,r}$, 
\[
\frac{1}{\mc{T}^k}\int_{\left\{ t_1,\ldots, t_k\right\}\in [0,\mc{T}]^k} T_m(t_{\alpha(m)})A_{m-1}T_{m-1}(t_{\alpha(m-1)})\ldots A_2T_2(t_{\alpha(2)})A_1T_1(t_{\alpha(1)}) f
\]
converges pointwise a.e..
\end{enumerate}
\end{Thm}

\end{document}